\documentclass[11pt]{amsart}
\usepackage{amsmath,cleveref,amssymb}

\title[Image of polynomials on graded matrix algebra]{A short note on the image of multilinear graded polynomials on matrix algebras}
\author[I.~da Silva]{Italo Bauer Rodrigues da Silva}
\address{Department of Mathematics, Instituto de Matem\'atica, Estat\'istica e Ci\^encia da Computa\c c\~ao, Universidade de S\~ao Paulo, SP, Brazil}
\email{italo.bauer@usp.br}
\thanks{I.~da Silva was financed in part by the Coordena\c c\~ao de Aperfei\c coamento de Pessoal de N\'ivel Superior - Brasil (CAPES) - Finance Code 001.}

\author[F.~Yasumura]{Felipe Yukihide Yasumura}
\address{Department of Mathematics, Instituto de Matem\'atica, Estat\'istica e Ci\^encia da Computa\c c\~ao, Universidade de S\~ao Paulo, SP, Brazil}
\email{fyyasumura@ime.usp.br}
\thanks{F.~Yasumura was supported by Fapesp grant no.~2024/14914-9 and 2025/16296-3.}

\newtheorem{Thm}{Theorem}
\newtheorem{lemma}[Thm]{Lemma}

\newtheorem{proposition}[Thm]{Proposition}

\theoremstyle{definition}

\theoremstyle{remark}

\begin{document}
\begin{abstract}
We investigate the subspaces obtained by images of multilinear graded polynomials evaluated on the matrix algebra endowed with the canonical $C_n$-grading, where $C_n$ denotes the cyclic group of order $n$. Moreover, this graded algebra admits a natural action of $C_n$ arising from a fine refinement of the grading. We prove that every $C_n$-submodule of the trivial component of the grading is the image of some multilinear polynomial in graded variables. In particular, we answer in the negative a recent conjecture posed by T.~de Castilho and L.~Centrone (2023).
\end{abstract}
\maketitle

\section{Introduction}
There are several recent papers devoted to the study of the images of (multilinear) polynomials evaluated on a given algebra of interest (see, for instance, \cite{AM1957,AEV2015,CF,FagMel2018,Fag2019,IvanMello,MaOl2016,Mello,Shoda,Spe2013}). This problem originates from a question posed by L'vov, asking whether the image of a multilinear polynomial evaluated on a matrix algebra is a vector subspace \cite[Problem 1.98]{Dn}. Positive results have been established for $2\times 2$ matrices \cite{KBSR2012}, while a partial answer is known for $3\times 3$ matrices \cite{KBSR2016}. This question is further motivated by an important counterpart in group theory (see \cite{Zel}).

In a similar direction, efforts have been made in the context of graded algebras, providing a new approach to the classical L'vov problem \cite{MG}. Also, the image of polynomials in multilinear graded variables evaluated on the matrix algebra endowed with the canonical grading (also known as the Vasilovsky grading) is studied in \cite{CM}. In that paper, the authors describe the linear span of the image of multilinear polynomials evaluated on $\mathrm{M}_n(\mathbb{Q})$, and conjecture that the same holds over any base field.

In the present paper, we investigate the same algebra. It is known that this algebra admits a natural structure of a $C_n$-module arising from a fine refinement of the grading. We prove that every $C_n$-submodule of the trivial component of the grading is the image of a multilinear polynomial. As a consequence, we answer the above conjecture in the negative.

\section{Notation and preliminaries}
Let $C_n = \langle \alpha \mid \alpha^n = 1 \rangle$, and let $\mathbb{F}$ be a field such that $\operatorname{char}\mathbb{F} = p \nmid n$. Initially, we assume that $\mathbb{F}$ contains a primitive $n$-th root of unity, but we will later drop this assumption.

Given an $\mathbb{F}$-algebra $\mathcal{A}$, a \emph{$C_n$-grading} on $\mathcal{A}$ is a vector space decomposition $\mathcal{A}=\bigoplus_{\tau\in C_n}\mathcal{A}_\tau$ such that $\mathcal{A}_\tau\mathcal{A}_{\tau'}\subseteq\mathcal{A}_{\tau\tau'}$, for all $\tau,\tau'\in C_n$. The subspaces $\mathcal{A}_\tau$ are called \emph{homogeneous components}, and their nonzero elements are called \emph{homogeneous elements}. Given $0\ne a\in\mathcal{A}_\tau$, its \emph{homogeneous degree} is $\deg_{C_n} a=\tau$.

We denote by $\mathrm{M}_n(\mathbb{F})$ the algebra of $n\times n$ matrices over $\mathbb{F}$, and by $e_{ij}$ the matrix unit, i.e., the matrix having a single nonzero entry equal to $1$ in position $(i,j)$. This algebra admits a natural $C_n$-grading given by
\begin{equation}\label{grading}
\left(\mathrm{M}_n(\mathbb{F})\right)_\tau:=\mathrm{Span}\{e_{ij}\mid\alpha^{i-j}=\tau\},\quad\tau\in C_n.
\end{equation}
This grading is called the \emph{canonical grading}, or the \emph{Vasilovsky grading}, on the matrix algebra (see \cite{V}).

We identify $\alpha$ with the cycle $(1\ 2\ \cdots\ n)\in\mathcal{S}_n$, where $\mathcal{S}_n$ denotes the symmetric group on $n$ symbols. Thus, $C_n\subseteq\mathcal{S}_n$ acts naturally on the set ${1,2,\ldots,n}$. The algebra $\mathrm{M}_n(\mathbb{F})$ then acquires a natural structure of a $C_n$-module: for $\tau\in C_n$, define
\begin{equation}\label{action}
\tau\cdot e_{ij}:=e_{\tau(i),\tau(j)}.
\end{equation}
Note that the grading is compatible with this action, in the sense that $\sigma\left(\mathrm{M}_n(\mathbb{F})\right)_\tau\subseteq\left(\mathrm{M}_n(\mathbb{F})\right)_\tau$, for each $\tau\in C_n$.

The free $C_n$-graded associative algebra is the free associative algebra $\mathbb{F}\langle X^{C_n}\rangle$, freely generated by $X^{C_n}=\{x_i^{(\tau)}\mid i\in\mathbb{N},\,\tau\in G\}$. The homogeneous degree of a monomial $x_{i_1}^{(\tau_1)}\cdots x_{i_m}^{(\tau_m)}$ is $\tau_1\cdots\tau_m$. This algebra satisfies the following universal property: any map $X^{C_n}\to\mathcal{A}$ that respects the $C_n$-grading, where $\mathcal{A}$ is a $C_n$-graded algebra, extends uniquely to a graded algebra homomorphism $\mathbb{F}\langle X^{C_n}\rangle\to\mathcal{A}$. In particular, given a polynomial $f(x_1^{(\tau_1)},\ldots,x_m^{(\tau_m)})$ and elements $a_1\in\left(\mathrm{M}_n(\mathbb{F})\right)_{\tau_1}, \ldots, a_m\in\left(\mathrm{M}_n(\mathbb{F})\right)_{\tau_m}$, we may evaluate $f(a_1,\ldots,a_m)$. We shall use the letter $z$ to denote variables of homogeneous degree $\alpha$, that is, $z_i := x_i^{(\alpha)}$ for each $i\in\mathbb{N}$.

\section{Main result}
Given $\chi\in\widehat{C_n}$, we let
$$
e_\chi:=\sum_{\tau\in C_n}\chi^{-1}(\tau)\tau(e_{ii}) =\sum_{i = 1}^n\chi^{-1}(\alpha^{i-1})e_{ii}\in\left(\mathrm{M}_n(\mathbb{F}\right))_1,
$$
and denote $\mathcal{V}_\chi:=\mathrm{Span}_\mathbb{F}\{e_\chi\}$. It is clear that $\mathcal{V}_\chi$ is an irreducible $C_n$-representation with character $\chi$. We set
\begin{equation*}
f_\chi(z_1,\ldots,z_n)=\frac{1}{n}\sum_{\tau\in C_n}\chi^{-1}(\tau)z_{\tau(1)}\cdots z_{\tau(n)},
\end{equation*}
where $deg(z_i) = \alpha$, for each $i$.

For each $i=1,\ldots,n-1$, we set $b_i:=e_{i,i+1}$, and $b_n:=e_{n1}$.

\begin{lemma}\label{lem1}
$f_\chi(b_{i_1},\ldots,b_{i_n})\ne0$ if and only if $(i_1,\ldots,i_n)$ is a cyclic permutation of $(1,\ldots,n)$. In addition, $f_\chi(b_{i_1},\ldots,b_{i_n})\in\mathcal{V}_\chi$.
\end{lemma}

\begin{proof}
From the definition of $b_i$, it follows that: 
\[
b_1b_2\dots b_n = e_{12}e_{23}\dots e_{n-1, n}e_{n, 1} = e_{11}
\]
Then $\alpha^k(b_1b_2\dots b_n) = \alpha^k(e_{11}) = e_{k+1, k+1}$. Then, $f(b_1,\ldots,b_n)=e_\chi\ne0$. Now, assume that $(i_1,\ldots,i_n)$ is a cyclic permutation of $(1,2,\ldots,n)$, say $(\tau(1),\tau(2),\ldots,\tau(n))$, for some $\tau\in C_n$. Then,
\[
f_\chi(b_{i_1},\ldots,b_{i_n})=f_\chi(\tau(b_1),\ldots,\tau(b_n))=\tau(f_\chi(b_1,\ldots,b_n))=\tau(e_\chi)=\chi(\tau)e_\chi.
\]

Conversely, suppose that $f_\chi(b_{i_1}, \dots , b_{i_n}) \not = 0$. Then there exists some $\tau \in C_n$ such that $b_{\tau(i_1)}b_{\tau(i_2)}\dots b_{\tau(i_n)} \not = 0$.

The multiplication rule of the $b_i$ is
\[
b_{i_l}b_{i_k} = \begin{cases} 
e_{{i_l},i_{l+2}}, & \text{if} \  i_l \equiv i_k - 1 \pmod{n} \\ 
0, & \text{otherwise} 
\end{cases}
\]
Thus, we conclude that the sequence $(i_1,i_2, \dots, i_n)$ is a cyclic permutation of $(1,2,\ldots,n)$.
\end{proof}

\begin{lemma}\label{im=v}
$\operatorname{Im} f_\chi = \mathcal{V}_\chi$.
\end{lemma}

\begin{proof}
From \Cref{lem1}, we have $0\ne\operatorname{Im} f_\chi \subseteq \mathcal{V}_\chi$. Since $\mathcal{V}_\chi$ is $1$-dimensional, we get $\operatorname{Im} f_\chi=\mathcal{V}_\chi$.
\end{proof}

\begin{lemma}\label{1}
Define
\[
r_\chi = \sum_{i=1}^n \chi^{-1}(\alpha^{i-1})\, b_i,\quad\,r_1 = \sum_{t=1}^n b_t,\quad \chi\in\widehat{C_n}.
\]
Then, $f_\chi(r_{\chi'}, r_1, \dots, r_1) = \delta_{\chi,\chi'}e_\chi$, where $\delta_{\chi,\chi'}=0$, if $\chi\ne\chi'$, and $\delta_{\chi,\chi}=1$.
\end{lemma}
\begin{proof}

First, note that $r_1^k=\sum_{i=1}^n e_{i,i+k}$, where the indexes should be understood as equality modulo $n$. Then,
\[
r_1^k r_\chi r_1^{n-1-k}
= \sum_{i=1}^n \chi^{-1}(\alpha^{i-1}) e_{i-k,i+1+n-1-k}
= \chi(\alpha^{k})e_\chi.
\]
Hence,
\[
f_\chi(r_\chi,r_1,\ldots,r_1)=\frac1n\sum_{k=0}^{n-1}\chi^{-1}(\alpha^{k})r_1^kr_\chi r_1^{n-1-k}=e_\chi.
\]
It is clear that, if $\chi'\ne\chi$, then $f_\chi(r_{\chi'},r_1,\ldots,r_1)=0$.
\end{proof}

As a consequence, we obtain that each $C_n$-submodule of $\left(\mathrm{M}_n(\mathbb{F})\right)_1$ is the image of some multilinear polynomial:
\begin{proposition}\label{prop}
Assume that $\mathbb{F}$ contains a primitive $n$-th rooty of the unity, let $\{\chi_1,\ldots,\chi_t\}\subseteq\widehat{C_n}$ and set
$$
f=\sum_{\ell=1}^tf_{\chi_\ell}(z_1,\ldots,z_n).
$$
Then, $\mathrm{Im}\,f=\mathcal{V}_{\chi_1}+\cdots+\mathcal{V}_{\chi_t}$.
\end{proposition}
\begin{proof}
It follows from \Cref{im=v,1}.
\end{proof}

From now on, we assume that $\mathbb{F}$ is an arbitrary field with either zero characteristic or 
$\operatorname{char}(\mathbb{F})$
does not divide $n$. Let $\mathbb{E}$ be a splitting field of $C_n$ over $\mathbb{F}$. Then, each element of the character group $\widehat{C_n}$ has values in $\mathbb{E}^\times$, where $\mathbb{E}^\times$ denotes the set of nonzero elements of $\mathbb{E}$. Moreover, the Galois group $\mathrm{Gal}(\mathbb{E}/\mathbb{F})$ acts on each character via
$$
g\cdot\chi:=g\circ\chi,\quad\chi\in\widehat{C_n},\,g\in\mathrm{Gal}(\mathbb{E}/\mathbb{F}).
$$

Now, write
\begin{equation}\label{eq}
\mathbb{F}C_n=\mathbb{F}C_ne_1\oplus\cdots\oplus\mathbb{F}C_ne_m,
\end{equation}
where each $e_i$ is a primitive central idempotent. Thus, $\mathbb{F}C_ne_i$ is a simple algebra, i.e., $\mathbb{F}C_ne_i$ is a field extension of $\mathbb{F}$ and a subfield of $\mathbb{E}$. For each $i$, after extending scalars, we have
\[
\mathbb{E}C_ne_i=\mathbb{E}C_ne_{\chi_{j_{1,i}}}\oplus\cdots\oplus\mathbb{E}C_ne_{\chi_{j_{m_i,i}}},
\]
for some characters $\chi_{j_{k,i}}$. The set of characters $\mathcal{O}_i:=\{\chi_{j_{1,i}},\ldots,\chi_{j_{m_i,i}}\}$ coincides with a $\mathrm{Gal}(\mathbb{E}/\mathbb{F})$-orbit inside $\widehat{C_n}$. In addition, the primitive central idempotent is precisely the sum
\[
e_{\mathcal{O}_i}=\sum_{\chi \in \mathcal{O}_i} e_\chi.
\]

We shall identify $\left(\mathrm{M}_n(\mathbb{F})\right)_1$ with the regular $C_n$-representation $\mathbb{F}C_n$. Using all the above notation, we have:

\begin{lemma}\label{coef}
Let $\mathcal{O}$ be a $\mathrm{Gal}(\mathbb{E}/\mathbb{F})$-orbit in $\widehat{C_n}$ and let
\[
f_{\mathcal{O}}(z_1,\ldots,z_n):=\sum_{\chi\in\mathcal{O}}f_\chi(z_1,\ldots,z_n).
\]
Then $f_{\mathcal{O}}\in\mathbb{F}\langle X^{C_n}\rangle$. In addition, $\mathrm{Im}\,f_{\mathcal{O}}=\mathbb{F}C_ne_\mathcal{O}$, where $e_\mathcal{O}=\sum_{\chi\in\mathcal{O}}e_\chi$.
\end{lemma}
\begin{proof}
By definition,
\[
f_\mathcal{O}=\sum_{\chi\in\mathcal{O}}\sum_{\tau\in C_n}\chi(\tau)^{-1}z_{\tau(1)}\cdots z_{\tau(n)}=\sum_{\tau\in C_n}\left(\sum_{\chi\in\mathcal{O}}\chi(\tau)^{-1}\right)z_{\tau(1)}\cdots z_{\tau(m)}.
\]
Since $\mathcal{O}$ is a $\mathrm{Gal}(\mathbb{E}/\mathbb{F})$-orbit, fixing any $\chi\in\mathcal{O}$, we have
\[
\sum_{\chi\in\mathcal{O}}\chi(\tau)^{-1}=\frac{|\mathcal{O}|}{|\mathrm{Gal}(\mathbb{E}/\mathbb{F})|}\sum_{g\in\mathrm{Gal}(\mathbb{E}/\mathbb{F})}(g\cdot\chi)(\tau)^{-1}.
\]
Hence, $\sum_{g\in\mathrm{Gal}(\mathbb{E}/\mathbb{F})}g(\chi(\tau)^{-1})\in\mathbb{E}^{\mathrm{Gal}(\mathbb{E}/\mathbb{F})}=\mathbb{F}$. As a consequence, $f_\mathcal{O}\in\mathbb{F}\langle X^{C_n}\rangle$.

Now, consider the notation of \Cref{1}, and let $r_\mathcal{O}=\sum_{\chi\in\mathcal{O}}r_\chi$. Then, from \Cref{1}, $f_\mathcal{O}(r_\mathcal{O},r_1,\ldots,r_1)=e_\mathcal{O}$. Moreover, for each $\tau\in C_n$, we have
\[
\tau(e_\mathcal{O})=f_\mathcal{O}(\tau(r_\mathcal{O}),\tau(r_1),\ldots,\tau(r_1))=f_\mathcal{O}(\tau(r_\mathcal{O}),r_1,\ldots,r_1).
\]
Let $\tau_1$, \dots, $\tau_s\in C_n$ be such that $\{\tau_1(e_\mathcal{O}),\ldots,\tau_s(e_\mathcal{O})\}$ is a $\mathbb{F}$-basis of $\mathbb{F}C_ne_{\mathcal{O}}$. Then, for $\lambda_1$, \dots, $\lambda_s\in\mathbb{F}$, one has
\[
f_\mathcal{O}(\sum_{j=1}^s\lambda_j\tau_j(r_\mathcal{O}),r_1,\ldots,r_1)=\sum_{j=1}^s\lambda_j\tau_j(e_\mathcal{O}).
\]
This gives $\mathrm{Im}\,f_\mathcal{O}=\mathbb{F}C_ne_\mathcal{O}$.
\end{proof}

Now, we obtain our main result.
\begin{Thm}
Let $n\in\mathbb{N}$ and $\mathbb{F}$ an arbitrary field, either of zero characteristic or $\mathrm{char}\,\mathbb{F}$ not dividing $n$. Consider the canonical $C_n$-grading on $\mathrm{M}_n(\mathbb{F})$, i.e., the one given by \eqref{grading}, and let $\mathcal{V}$ be a $C_n$-submodule of $(\mathrm{M}_n(\mathbb{F}))_1$, where the $C_n$-action is given by \eqref{action}. Then, there exists a multilinear polynomial in graded variables $f\in\mathbb{F}\langle X^{C_n}\rangle$ such that $\mathrm{Im}(f)=\mathcal{V}$.
\end{Thm}
\begin{proof}
Each $C_n$-submodule $\mathcal{V}$ of $\left(\mathrm{M}_n(\mathbb{F})\right)_1$ is a direct sum of components appearing in the decomposition~\eqref{eq}, say $\mathcal{V}=\mathbb{F}C_ne_{\mathcal{O}_1}\oplus\cdots\oplus\mathbb{F}C_ne_{\mathcal{O}_t}$. Let $f_{\mathcal{O}_i}$ be as in \Cref{coef}, for each $i$, and let $f=\sum_if_{\mathcal{O}_i}$. Let $\tau_{1,i}$, \dots, $\tau_{s_i,i}\in C_n$ be such that $\{\tau_{1,i}(e_{\mathcal{O}_i}),\ldots,\tau_{s_i,i}(e_{\mathcal{O}_i})\}$ is a $\mathbb{F}$-basis of $\mathbb{F}C_ne_{\mathcal{O}_i}$. Then, using the notation of the proof of \Cref{coef} and \Cref{1}, one has, for all $\lambda_{i,j}\in\mathbb{F}$,
\[
f(\sum_{i=1}^t\sum_{j=1}^{s_i}\lambda_{i,j}\tau_{j,i}(r_{\mathcal{O}_i}),r_2,\ldots,r_n)=\sum_{i=1}^t\sum_{j=1}^{s_i}\lambda_{i,j}\tau_{j,i}(e_{\mathcal{O}_i}).
\]
Hence, $\mathrm{Im}\,f=\mathcal{V}$.
\end{proof}
This provides a negative answer to \cite[Conjecture 1]{CM}.

Finally, it is interesting to ask whether if a image of a multilinear polynomial is a vector space. Given a permutation $\sigma\in\mathcal{S}_{n-1}$ and $\chi\in\widehat{C_n}$, we set
\[
f_{\sigma,\chi}:=f_\chi(z_{\sigma(1)},\ldots,z_{\sigma(n-1)},z_n).
\]
Then, it is clear that $\{f_{\sigma,\chi}\mid\sigma\in\mathcal{S}_{n-1},\chi\in\widehat{C_n}\}$ is a basis of the multilinear polynomials of degree $n$ in the variables $\{z_1,\ldots,z_n\}$. We provide a positive answer for the case where $n=3$ and the base field contains a primitive cube root of the unity:

\begin{proposition}
    Let $\mathbb{F}$ be a field containing a primitive cube root of unity (so, $\mathrm{char}\,\mathbb{F}\ne3$), and let $f \in \mathbb{F}\langle X^{C_3} \rangle$ be a multilinear polynomial in variables $z_1$, $z_2$, $z_3$ such that $\deg(z_i) = \alpha$ for $i = 1, 2, 3$. Then, $\mathrm{Im}(f)$, evaluated on $\mathrm{M}_3(\mathbb{F})$ endowed with the canonical $C_3$-grading, is a vector subspace.
\end{proposition}

\begin{proof}
As the above discussion, write
    \[
    f = 
    \sum_{\sigma \in\mathcal{S}_2}\sum_{\chi \in \widehat{C_3}}\nu_{\sigma, \chi}f_{\sigma, \chi}.
    \]

    From Lemma~\ref{1} and its notation, we have
    \[f(r_1,r_1,\sum_{\chi\in\widehat{C_3}}\lambda_\chi r_\chi) = \sum_{\chi \in \widehat{C_3} }\lambda_\chi\chi(\alpha)\left(\sum_{\sigma \in S_2}\nu_{\sigma, \chi}\right) e_\chi.
    \]
    This evaluation establishes the result for all cases, except for the scenario where there exists at least one pair $\{\nu_{(1), \chi}, \nu_{(12), \chi}\}$ such that both are non-zero and their sum vanishes.

    To address this case, we analyze the polynomial
    \[
    f = \sum_{\chi \in \widehat{C_n}}\nu_\chi (f_\chi - f_{(12), \chi}).
    \]
    The next two evaluations shows that the image is always a vector subspace:
    \[
    \begin{split}
    f(\lambda_0r_{\chi_0} + \lambda_1 r_{\chi_1}, \lambda_2 r_{\chi_1}, r_{\chi_2}) = {} & \lambda_0\lambda_2\nu_{\chi_0}(\chi_1^{-1}(\alpha^2)-1)e_{\chi_0} \\
    & + \lambda_0\nu_{\chi_1}(\chi_2^{-1}(\alpha^2)- \chi_2^{-1}(\alpha))e_{\chi_1} \\
    & + \lambda_1\nu_{\chi_2}(\chi_2^{-1}(\alpha) - \chi_2^{-1}(\alpha^2))e_{\chi_2},
    \end{split}
    \]
    
    \[
    \begin{split}
    f(\lambda_0r_{\chi_0} + \lambda_2 r_{\chi_2}, r_{\chi_1}, r_{\chi_2}) = {} & \lambda_0\nu_{\chi_0}(\chi_1^{-1}(\alpha^2)-\chi_1^{-1}(\alpha))e_{\chi_0} \\
    & + \lambda_2\nu_{\chi_2}(\chi_2^{-1}(\alpha^2)- \chi_2^{-1}(\alpha))e_{\chi_2}.
    \end{split}
    \]
    The proof is complete.
\end{proof}

\end{document}